\documentclass[12pt]{amsart}
\usepackage{times}
\usepackage{amssymb}
\usepackage[dvips]{graphicx}

\allowdisplaybreaks
\usepackage{amsmath,amssymb,txfonts}
\usepackage{amssymb}
\usepackage{amsxtra}
\usepackage{amsmath}
\usepackage{txfonts}
\usepackage{amsmath,amssymb,txfonts}
\usepackage{amssymb}
\usepackage{amsxtra}
\usepackage{amsmath}
\usepackage{txfonts}
\usepackage{amsmath,amssymb,txfonts}
\usepackage{amssymb}
\usepackage{amsxtra}
\usepackage{amsmath}
\usepackage{txfonts}
\usepackage{color}
\usepackage[cp1252]{inputenc}

\usepackage{mathrsfs}
\usepackage{graphicx}

\usepackage[active]{srcltx}
    \usepackage{amsmath,amssymb}

    \usepackage{latexsym}

    \usepackage{CJK}
    \newtheorem{definition}{Definition}
    \newtheorem{lemma}{Lemma}
    
    \newtheorem{theorem}{Theorem}
    \newtheorem{proposition}{Proposition}
    \newcommand{\R}{{\mathbb R}}

\begin{document}

\title{The Affine BV-capacity}

\author[T. Wang]{T. Wang}
\address{Department of Mathematics and Information Science, Shaanxi Normal University, West Chang'an Street 620, 710119 Xi'an, Shaanxi, China  }
\email{tuowang734@gmail.com}
\thanks{Project supported by the Recruitment Program for Young Professionals}
\author[J. Xiao]{J. Xiao}
\address{Department of Mathematics and Statistics, Memorial University, St. John's, NL A1C5S7, Canada}
\email{jxiao@mun.ca}
\thanks{Project supported by NSERC of Canada as well as by URP of Memorial University, Canada.}
\begin{abstract}{This paper is devoted to a geometric-measure-theoretic study of the brand new affine BV-capacity which is essentially different from the classic BV-capacity in dimension greater than one.}
	\end{abstract}

\maketitle

\begin{CJK*}{GBK}{li}

\section*{Introduction}\label{s0}

As is well known, a function of bounded variation, simply a BV-function, is a real-valued function whose total variation is finite. In one variable, a function defined on an open interval of bounded variation is just a function with respect to which we can find the Riemann-Stieltjes integral of a continuous function on the interval. In more-than-one variables, a function defined on an open subset of $\mathbb R^{n\ge 2}$ is said to have bounded variation provided that its distributional derivative is a vector-valued finite Radon measure over the subset.

Importantly, the BV-functions form an algebra of discontinuous functions whose first derivative exists almost everywhere - thanks to this nature, this algebra is frequently utilized to define generalized solutions of nonlinear problems involving functional analysis, ordinary and partial differential equations, mathematical physics and engineering.

Even more importantly, the BV-functions can be used to define the perimeter $P_{BV}(E)=\|1_E\|_{BV(\mathbb R^n)}$ (cf. Section \ref{s11}) of an arbitrary set $E$ in $\mathbb R^n$ with the indicator $1_E$, and then to establish the classical isoperimetric inequality comparing the volume (or Lebesgue measure) $V(E)$ of $E$ and $P_{BV}(E)$:
\begin{equation}
\label{e01} \left(\frac{V(E)}{\omega_n}\right)^\frac{1}{n}\le \left(\frac{P_{BV}(E)}{n\omega_n}\right)^\frac{1}{n-1},
\end{equation}
where
$$
\omega_k=\frac{\pi^\frac{k}{2}}{\Gamma(1+\frac{k}{2})}\quad\hbox{and}\quad k\omega_k=\sigma_{k-1}
$$
are the volume of the unit ball $\mathbb B^k$ and the surface area of the unit sphere $\mathbb S^{k-1}$ in the $1\le k\le n$ dimensional Euclidean space $\mathbb R^k$, respectively - of course - $\Gamma(\cdot)$ is the standard gamma function. As an optimal consequence of \cite[Theorem 5.12.4]{Zi} (i.e. \cite[Theorem 4.7]{MZ}) where $\mu$ is taken as the $n$-dimensional Lebesgue measure on $\mathbb R^n$, not only \eqref{e01} is equivalent to the sharp BV-Sobolev inequality below:
\begin{equation}
\label{e02}
\frac{\|f\|_{L^\frac{n}{n-1}(\mathbb R^n)}}{\omega_n^{1-\frac1n}}\le \frac{\|f\|_{BV(\mathbb R^n)}}{n\omega_n},
\end{equation}
but also \eqref{e02} is equivalent to the following isocapacitary inequality linking $V(E)$ and the BV-capacity $C_{BV}(E)$ of $E$:
\begin{equation}
\label{e03}
\left(\frac{V(E)}{\omega_n}\right)^\frac{1}{n}\le \left(\frac{C_{BV}(E)}{n\omega_n}\right)^\frac{1}{n-1},
\end{equation}
where
$$
C_{BV}(E)=\inf\left\{\|f\|_{BV(\mathbb R^n)}: f\in BV(\mathbb R^n)\ \&\ f\ge 1\ \ \hbox{in}\ \ \hbox{a\ neighborhood\ of}\ E\right\}.
$$
Yet, from affine geometric perspective the BV-functions have been further exploited in \cite{Wang}, as an essential extension of \cite{Z} (whose partial consequences are \cite{CLYZ, LYZ00, LYZ02, LYZb, HS, HSX, Zh}), to establish the following affine Sobolev inequality (cf. Section \ref{s12}) which is stronger than \eqref{e02}:
\begin{equation}
\label{e04}
\frac{\|f\|_{L^\frac{n}{n-1}(\mathbb R^n)}}{\omega_n^{1-\frac1n}}\le\frac{\left(\int_{\mathbb S^{n-1}}\left(\int_{\mathbb R^n}|u\cdot\sigma_f|\,d|Df|\right)^{-n}\,\frac{du}{n\omega_n}\right)^{-\frac1n}}{2\omega_{n-1}}.
\end{equation}
Meanwhile, \eqref{e04} amounts to the following affine isoperimetric inequality connecting $V(E)$ and the volume $V(\Pi^\circ E)$ of the polar body $\Pi^\circ E$ of the projection body $\Pi E$ (cf. Section \ref{s11}):
\begin{equation}
\label{e05}
\big(V(E)\big)^{n-1}V(\Pi^\circ E)\le\big(\omega_n\omega^{-1}_{n-1}\big)^n\quad\hbox{i.e.}\quad
 \left(\frac{V(E)}{\omega_n}\right)^\frac1n\le \left(\frac{P_{BV,d}(E)}{2\omega_{n-1}}\right)^\frac1{n-1},
 \end{equation}
which is stronger than \eqref{e01}. For the related definitions see Section \ref{s1}. Surprisingly, as proved in \cite{X15} whose relatives are \cite{X07, X14, X14r, XZ}, the inequality \eqref{e05} for any compact domain $E$ with piece-wise $C^1$ boundary is equivalent to the affine isocapacitary inequality
\begin{equation}
\label{e06}
\left(\frac{V(M)}{\omega_n}\right)^\frac{1}{n}\le\left(\frac{C_{1,d}(M)}{2\omega_{n-1}}\right)^\frac1{n-1}
\end{equation}
for any compact domain $M\subset\mathbb R^n$ with piece-wise $C^1$ boundary, where
$$
C_{1,d}(K)=\inf\left\{\left(\int_{\mathbb S^{n-1}}\|u\cdot\nabla f\|_{L^1(\mathbb R^n)}^{-n}\,\frac{du}{n\omega_n}\right)^{-\frac1n}:\ f\in C_c^1(\mathbb R^n)\ \&\ f\ge 1_K\right\}
$$
is the affine $1$-capacity of a compact set $K\subset\mathbb R^n$ and $C_c^1(\mathbb R^n)$ denotes compactly supported functions as usual. Importantly, one has
\begin{equation}
\label{e07}
\frac{C_{1,d}(K)}{2\omega_{n-1}}\le\frac{C_{BV}(K)}{n\omega_n}\quad\forall\quad\hbox{compact}\ K\subset\mathbb R^n,
\end{equation}
which in turn implies that \eqref{e06} is stronger than \eqref{e03} for any compact domain $E$ with piece-wise $C^1$-boundary.

Now, two natural questions to ask are: 1) Can we define the affine BV-capacity $C_{BV,d}$ which generalizes $C_{1,d}$? 2) How much does $C_{BV,d}$ geometrically behave like $C_{BV}$? Needless to say, a settlement of this issue needs an innovative application of ideas from both convex geometry and calculus of variation. Therefore some fundamental materials from convex geometry and BV theory are collected in Section \ref{s1}. In Section \ref{s3}, the definitions and basic properties of $C_{BV,d}$ are given. In Section \ref{s4}, we prove that $C_{BV,d}$ does not increase under the Steiner symmetrization and its rounding. In Section \ref{s5} we show that the last mutual equivalence with rough constants can be recovered from a consideration of the affine traces involving $C_{BV,d}$, thereby investigating the affine $[1,\frac{n}{n-1})\ni q$-Cheeger constant.

\section{Preliminaries}\label{s1}

 \subsection{Basics regarding convex bodies}\label{s11}
 For quick later reference we collect some notation and basic facts about convex bodies,
 (see, e.g., \cite{G,S}).

For each integer $k=0,1,2,..., n$, the symbol $H^k$ stands for the $k$-dimensional Hausdorff measure. The standard inner product of the
 vectors $x,y\in\mathbb R^n$ is denoted by $x\cdot y$ and $|x|=\left(x\cdot x\right)^{1/2}$ is the Euclidean norm. We write $\mathbb S^{n-1} = \left\{x\in\mathbb R^n:  |x| = 1\right\}$ for the boundary
 of the Euclidean unit ball $\mathbb B^n$ in $\mathbb R^n$. We denote the area of unit sphere $\mathbb S^{n-1}$ by $n\omega_n$ and the volume of the Euclidean ball $\mathbb B^n$ by $\omega_n.$ We denote the ball with center $x\in\R^n$ and radius $r>0$ by $B_r(x)$.

 A convex body is a compact convex subset of $\mathbb R^n$ with non-empty interior. Each non-empty compact convex set $K$ is uniquely determined by its support function $h_K$, defined by
 $$
 h_K(x)=\max\left\{x\cdot y: y\in K\right\}\quad\forall\quad x\in\mathbb R^n.
 $$
 Note that $h_K$ is positively homogeneous of degree $1$ and sub-additive. Conversely, every function with these two properties is the support function of a unique compact convex set. We write $\mathcal K^n$ for the set of convex bodies in $\mathbb R^n$ endowed with the Hausdorff metric which is the metric induced by the maximum norm of the support function,
 by $\mathcal K^n_0$ the set of convex bodies containing origin in the interior.  If $K\in \mathcal K^n_0$ , then the polar body $K^\circ$ of $K$ is defined by
 $$
 K^\circ=\left\{x\in \mathbb R^n: x\cdot y\leq 1\quad\forall\quad y\in K\right\}.
 $$
 Let us denote the standard othonormal basis of $\mathbb R^n$ by $\left\{e_1,e_2,\cdots,e_n\right\}.$ Then the Steiner symmetral $S_{e_n}(K)$ of a set $K$ with respect to $e_n$ is defined by
 \begin{equation}\label{steiner symmetral}
 S_{e_n}(K)=\left\{x+te_n:-\frac{H^1(L_x^{e_n}\cap K)}{2}\leq t \leq \frac{H^1(L_x^{e_n}\cap K)}{2} \quad\&\quad x\in \mathbb R^{n-1}\times\left\{0\right\} \right\},
 \end{equation}
 where $L_x^{e_n}$ is the line passing through $x$ in the direction $e_n.$ The rounding $R(K)$ of a set $K$ is defined by
 \begin{equation}\label{rounding}
 R(K)=B_{r_K}(0)\ \ \mbox{~with~}\ \ V(B_{r_K}(0))=V(K).
 \end{equation}
 Let $K$ be a convex body in $\mathbb R^n$ and $\nu: \partial K \rightarrow \mathbb S^{n-1} $ the generalized Gauss map. For each Borel set $\omega\subset \mathbb S^{n-1}$, the inverse spherical image $\nu^{-1}(\omega)$ of $\omega$ is the set of all boundary points of $K$ which have an outer unit normal belonging to the set $\omega$ and we know from \cite{S} that $\nu^{-1}(\omega)$ is a measurable set. Associated with each convex body $K\in\mathcal K^n$ is a Borel measure $S_K(\cdot)$ on $\mathbb S^{n-1}$ called the surface area measure of $K$, defined by
 $$
 S_K(\omega) = H^{n-1}(\nu^{-1}(\omega)),
 $$
 for each Borel set $\omega\subset \mathbb S^{n-1}$, that is, $S_K(\omega)$ is the $(n-1)$-dimensional Hausdorff measure of the set of all points on $\partial K$ where some outer normal unit vector lies in $\omega$. The surface area of a convex body $K$ is defined to be
 $$
 S_K(\mathbb S^{n-1})=H^{n-1}(\partial K).
 $$
 Projection bodies were introduced by Minkowski at the turn of the last century and have proved to be very useful in many ways and subjects. They are defined in the following way. The projection body $\Pi K$ of $K\in\mathcal K^n$ is the convex body whose support function is given by
 $$
 h_{\Pi K}(v)=\frac{1}{2}\int_{\mathbb S^{n-1}}|u\cdot v|dS_K(u)\quad\forall\quad v\in\mathbb R^n.
 $$
 The projection operator $\Pi$ has strong contravariance and invariance properties: for all $\phi\in GL(n)$ and translation $\tau$, we have
 $$
 \Pi(\phi K)=|\det\phi|\phi^{-t}\Pi K\quad\&\quad \Pi(\tau K)=\Pi K\quad\forall\quad K\in \mathcal K^n.
 $$

 The Petty projection inequality for convex bodies states:
 \begin{theorem}\label{t0p}
 For $K\in \mathcal K^n$, we have
 $$
 V(K)^{n-1}V(\Pi^\circ K) \leq (\omega_n\omega^{-1}_{n-1})^n
 $$
 with equality if and only if $K$ is an ellipsoid.
 \end{theorem}
 This inequality is found to be stronger than the classical isoperimetric inequality for convex bodies and yet both sides of the inequality are invariant under affine transformations. Motivated by this, the affine surface area of $K\in\mathcal K^n$ is defined by
 $$
 I_1(K)=2\big(\omega_n^{-1}V(\Pi^\circ K)\big)^{-\frac1n} =\left(\int_{\mathbb S^{n-1}}\Big(\int_{\mathbb S^{n-1}}|u\cdot v|\,dS_K(v)\Big)^{-n}\,\frac{du}{n\omega_n}\right)^{-\frac{1}{n}},
 $$
 which induces another representation of the Petty projection inequality in Theorem \ref{t0p}:
 $$
 \left(\frac{V(K)}{\omega_n}\right)^\frac1n\le \left(\frac{I_1(K)}{2\omega_{n-1}}\right)^\frac1{n-1}.
 $$
 We choose this normalization here for convenience. We note that the affine surface area is invariant under translation and $SL(n)$ transformation (see, e.g. \cite[p.570]{S}).

 \subsection{The space $BV(\R^n)$ and its induced perimeters}\label{s12}
 Referring to \cite{Setsoffiniteperimeterandgeometricvariational}, we use
 $$
 C_c^\infty(\mathbb R^n);\quad C_c^1(\mathbb R^n);\quad Lip_c(\mathbb R^n)
 $$
 to represent the class of all compactly supported infinitely differentiable functions; all compactly supported continuously differentiable functions; all compactly supported Lipschitz continuous functions in the Euclidean space $\mathbb R^{n\ge 1}$. We say that an $L^1(\mathbb R^n)$-function $f$ is of bounded variation on $\mathbb R^n$, written as $f\in BV(\mathbb R^n)$, provided that there is a vector-valued Radon measure $Df=(D_jf)_{j=1}^n$ such that
 $$
 \int_{\mathbb R^n}f(x)\,\partial_{x_j}\phi(x)\,dx=-\int_{\mathbb R^n}\phi\,dD_jf\quad\forall\ \phi\in C^1_c(\mathbb R^n)\ \&\ j\in\{1,...,n\}.
 $$
 The variation measure $|Df|$ of a Borel set $E\subset \mathbb R^n$ is defined by
 $$
 |Df|(E)=\sup\left\{\sum_{h=1}^{\infty}|Df(E_h)|:E_h\cap E_k=\emptyset\ \ \&\ \ \bigcup_{h=1}^{\infty} E_h\subset E\right\}.
 $$
 Note that for each $f\in BV(\mathbb R^n)$, $Df$ has the Radon-Nikodym derivative $\sigma_f$ with respect to the non-negative Radon measure $|Df|$. So one has the following divergence formula:
 \begin{equation}
 \label{e11}
 \int_{\mathbb R^n}f(x)\,\hbox{div}\psi(x)\,dx=-\int_{\mathbb R^n}\psi\cdot\sigma_f\,d|Df|\quad\forall\ \psi=(\psi_j)_{j=1}^n\in C_c^1(\mathbb R^n)\times\cdots\times C_c^1(\mathbb R^n),
 \end{equation}
 whence defining the BV-norm
 $$
 \|f\|_{BV(\mathbb R^n)}=\sup_{\psi}\int_{\mathbb R^n}f(x)\,\hbox{div}\psi(x)\,dx
 $$
 where the supremum is taken over all
 $$
 \psi=(\psi_j)_{j=1}^n\in C_c^1(\mathbb R^n)\times\cdots\times C_c^1(\mathbb R^n)\quad\hbox{with}\quad\sup_{x\in\mathbb R^n}|\psi(x)|\le 1,
 $$
 and its affine variant discovered in \cite{Wang} by
 $$
 \|f\|_{BV,d(\mathbb R^n)}=\left(\int_{\mathbb S^{n-1}}\Big(\int_{\R^n}|u\cdot\sigma_f(x)|\,d|Df|\Big)^{-n}\,\frac{du}{n\omega_n}\right)^{-\frac{1}{n}}.
 $$
 In the above and below, $du$ represents the standard surface area measure on $\mathbb S^{n-1}$.

 Importantly, for any measurable set $E\subseteq\R^n$ we can define the perimeter $P_{BV}(E)=\|1_E\|_{BV(\R^n)}$ and its affine counterpart $P_{BV,d}(E)=\|1_E\|_{BV,d(\R^n)}$.
 A measurable set $E\subset\mathbb R^n$ has finite perimeter in $\R^n$ if
 $1_E\in BV(\R^n)$, i.e., $P_{BV}(E)<\infty$. For a set $E$ with finite perimeter, we say $x\in\partial^\star E,$ the reduced boundary of $E$, if
 \begin{itemize}
 \item[(i)] $|D1_E|(B_r(x))>0$ for all $r>0,$
 \item[(ii)] $\lim_{r\rightarrow0^+} \frac{D1_E(B_r(x))}{|D1_E|(B_r(x))}= \sigma_{1_E}(x),$
 \item[(iii)] $|\sigma_{1_E}(x)|=1.$
 \end{itemize}
 In this case we call the vector field $\nu_E =-\sigma_{1_E}$ the measure theoretic outer unit normal to $E$. Let $E$ be a set of finite perimeter, it is well known that
 $$
 D1_E = -\nu_EdH^{n-1}\lfloor\partial^{\star}E, ~~~|D1_E|(\R^n)=H^{n-1}(\partial^\star E).
 $$
 The projection body for sets of finite perimeter introduced in \cite{Wang} has proved to be a natural extension of the projection body for convex bodies. They are defined in the following way. The projection body $\Pi E$ with $1_E\in BV(\mathbb R^n)$ is the convex body whose support function is given by
 $$
 h_{\Pi E}(v)=\frac{1}{2}\int_{\partial^\star E}|\nu_E(x)\cdot v|dH^{n-1}(x)\quad\forall\quad v\in\mathbb R^n.
 $$
 Correspondingly, the affine surface area of $E$, i.e., the affine perimeter of $E$, is given by
 $$
 P_{BV,d}(E)=({2^n\omega_n})^{\frac{1}{n}}V^{-\frac{1}{n}}(\Pi^\circ E) =\left(\int_{\mathbb S^{n-1}}\Big(\int_{\partial^\star E}|\nu_E(x)\cdot u|\,dH^{n-1}(x)\Big)^{-n}\,\frac{du}{n\omega_n}\right)^{-\frac{1}{n}}.
 $$
 By the strict convergence topology we mean the topology induced by the distance
 $$
 d(f,g)=\int_{\R^n}|f(x)-g(x)|dx + \big||Df|(\R^n)-|Dg|(\R^n)\big|\quad\forall\quad f,g\in BV(\R^n).
 $$
 If $\rho$ is a smooth function on $\R^n$, satisfying the following three requirements
 \begin{itemize}
   \item [(i)] it is compactly supported,
   \item [(ii)] $\int_{\mathbb{R}^n}\rho(x)\mathrm{d}x=1,$
   \item [(iii)] $\lim_{\epsilon\to 0}\rho_\epsilon(x) = \lim_{\epsilon\to 0}\epsilon^{-n}\rho(x / \epsilon)=\delta(x).$
 \end{itemize}
 where $\delta(x)$ is the Dirac delta function and the limit is understood in the space of Schwartz distributions, then $\rho_{\epsilon}$ is an approximation to the identity.

Now, a combination of the H\"older inequality, the Fubini theorem, \eqref{e11} and \cite[Theorem 6.5]{Wang} implies
\begin{equation}
\label{eBVe}
 \left(\frac{2\omega_{n-1}}{n\omega_n}\right)\|f\|_{BV(\mathbb R^n)}\ge \left(\int_{\mathbb S^{n-1}}\Big(\int_{\R^n}|u\cdot\sigma_f(x)|\,d|Df|\Big)^{-n}\,\frac{du}{n\omega_n}\right)^{-\frac{1}{n}}\ge \left(\frac{2\omega_{n-1}}{\omega_n^{1-\frac1n}}\right){\|f\|_{L^\frac{n}{n-1}(\mathbb R^n)}}.
 \end{equation}
 So, we always have
 \begin{equation}
 \label{e12P}
 \frac{P_{BV,d}(E)}{2\omega_{n-1}}\le\frac{P_{BV}(E)}{n\omega_n}.
\end{equation}
This inequality \eqref{e12P} actually compares two perimeter-radii and hence is very natural due to the fact that
$$
2\omega_{n-1}=2H^{n-1}(\mathbb B^n\cap\mathbb R^{n-1})\quad\&\quad n\omega_n=H^{n-1}(\mathbb S^{n-1})
$$
and so $2\omega_{n-1}$ can be treated as the projection of $\sigma_{n-1}=n\omega_n$.

\section{Definitions and fundamentals of the affine BV-capacity}\label{s3}

\subsection{The original definition of $C_{BV,d}$}\label{s31} The concept of an affine BV-capacity is motivated by \cite[5.1]{Zi} (cf. \cite{HK, HKM, Maz}).

\begin{definition}\label{BVoriginaldefinition}
Given a subset $E\subset\R^n.$ Let $\mathscr A(E)$ be the class of all BV-functions $f$
with $f\geq1$ on a neighbourhood of $E$, and denote by
$$
C_{BV}(E)=\inf_{f\in\mathscr A(E)}\|f\|_{BV(\mathbb R^n)}\quad\hbox{and}\quad
C_{BV,d}(E)=\inf_{f\in\mathscr A(E)}\|f\|_{BV,d(\R^n)}
$$
the BV-capacity and the affine BV-capacity of $E$ respectively. From \eqref{eBVe} it follows that
\begin{equation}
\label{e12C}
\frac{C_{BV,d}(E)}{2\omega_{n-1}}\le\frac{C_{BV}(E)}{n\omega_n}
\end{equation}
always holds and both capacities coincide under $n=1$.
\end{definition}

\subsection{Two alternatives of $C_{BV,d}$ for compact sets}\label{s32} Thanks to Definition \ref{BVoriginaldefinition}, from now on the dimension $n$ is always assumed to be greater than $1$. Just like $C_{BV}$ (cf. \cite[5.12]{Zi} and \cite{Fl, LXZ}), a simple regularization argument and a geometric realization yield the following formulas for $C_{BV,d}(E)$ in case $E$ is compact.

\begin{theorem}\label{t21} Let $K$ be a compact subset of $\mathbb R^n$.

\begin{itemize}
\item[(i)] If $F(\mathbb R^n)$ is in $\Big\{C_c^\infty(\mathbb R^n), C_c^1(\mathbb R^n), Lip_c(\mathbb R^n)\Big\}$, then
\begin{equation}
\label{eq21}
C_{BV,d}(K)=\inf\left\{\left(\int_{\mathbb S^{n-1}}\|\nabla_uf\|_{L^1(\mathbb R^n)}^{-n}\,\frac{du}{n\omega_n}\right)^{-\frac{1}{n}}: f\in F(\mathbb R^n)\ \&\ f\geq 1_K\right\}.
\end{equation}
where $\nabla_u f=u\cdot\nabla f$ stands for the derivative of $f$ along the direction $u\in\mathbb S^{n-1}$ and $\|\cdot\|_{L^1(\R^n)}$ denotes the Lebesgue $1$-integral on $\mathbb R^n$.

\item[(ii)] If $\mathscr{B}(K)$ consists of all bounded open sets $O$ containing $K$, then
\begin{equation}
\label{eq22}
C_{BV,d}(K)=\inf_{O\in\mathscr B(K)}P_{BV,d}(O).
\end{equation}
 Moreover, if $K$ is convex then
\begin{equation}
\label{eq23}
C_{BV,d}(K)=P_{BV,d}(K).
\end{equation}
\end{itemize}
\end{theorem}

\begin{proof} (i) Due to a mollification (cf. \cite[5.2]{Pf}), it is enough to verify the result for ${F}(\R^n)=C_c^\infty(\mathbb R^n)$.

Since $C_c^\infty(\R^n)\subset BV(\mathbb R^n),$ we have
$$
C_{BV,d}(K)\leq \inf\left\{\left(\left(n\omega_n\right)^{-1}\int_{\mathbb S^{n-1}}\Big(\int_{\mathbb R^n}|u\cdot\sigma_f|\,d|Df|\Big)^{-n}\,du\right)^{-\frac{1}{n}}: f\in C_c^\infty(\R^n)\ \&\ f\geq 1_K\right\}=:C_{1,d}(K).
$$

To prove the reverse one of the last inequality, we use the standard approximation technique. According to Definition \ref{BVoriginaldefinition}, for any $\epsilon>0$ there exists an $f\in\mathscr A(K)$ with
$$
\left(\int_{\mathbb S^{n-1}}\Big(\int_{\mathbb R^n}|u\cdot\sigma_f|\,d|Df|\Big)^{-n}\,\frac{du}{n\omega_n}\right)^{-\frac{1}{n}}
<C_{BV,d}(K)+\epsilon,
$$
so there exists an approximation to the identity $\rho_{\epsilon}$ with $$
\left(\int_{\mathbb S^{n-1}}\|\nabla_u(f\ast\rho_{\epsilon})\|_1^{-n}\,\frac{du}{n\omega_n}\right)^{-\frac{1}{n}}\le
\left(\int_{\mathbb S^{n-1}}\Big(\int_{\mathbb R^n}|u\cdot\sigma_f|\,d|Df|\Big)^{-n}\,\frac{du}{n\omega_n}\right)^{-\frac{1}{n}}+\frac{\epsilon}{2}\le C_{BV,d}(K)+\frac{3}{2}\epsilon.
$$
Notice that for some small $t(\epsilon)>0$ we have $$\frac{f\ast\rho_{\epsilon}}{1-t(\epsilon)}\geq 1_K.
$$
Thus we have
$$
C_{1,d}(K)\le\left(\int_{\mathbb S^{n-1}}\Big\|\nabla_u\big(\frac{f\ast\rho_{\epsilon}}{1-t(\epsilon)}\big)\Big\|_{L^1(\R^n)}^{-n}\,\frac{du}{n\omega_n}\right)^{-\frac{1}{n}}\leq \big({1-t(\epsilon)}\big)^{-1}\left(C_{BV,d}(K)+\frac{3}{2}\epsilon\right).
$$
Upon letting $\epsilon \rightarrow 0,$ we get
$$C_{1,d}(K)\leq C_{BV,d}(K),$$
thereby reaching \eqref{eq21}.

(ii) On the one hand, given a function $f\in\mathscr{A}(K)$ and $t>0$, let
$$
O_t(f)=\{x\in\mathbb R^n: |f(x)|>t\}.
$$
Then, from Definition \ref{BVoriginaldefinition} and the Minkowski inequality and the fact that $O_t(f)$ is a set of finite perimeter for almost all $t\in\mathbb R$, it follows that there exists an $f\in\mathscr A(K)$ such that
\begin{eqnarray*}
C_{BV,d}(K)+\epsilon&\geq&\left(\int_{\mathbb S^{n-1}}(\int_{\R^n}|u\cdot\sigma_f(x)|d|Df|)^{-n}\,\frac{du}{n\omega_n}\right)^{-\frac{1}{n}}\\
&=&\left(\int_{\mathbb S^{n-1}}\Big(\int_0^\infty\int_{\partial O_t(f)}|u\cdot\sigma_f|\,dH^{n-1}dt\Big)^{-n}\,\frac{du}{n\omega_n}\right)^{-\frac{1}{n}}\\
&\ge&\int_0^\infty\left(\int_{\mathbb S^{n-1}}\Big(\int_{\partial^\star O_t(f)}|u\cdot\nu_{O_t(f)}|\,dH^{n-1}\Big)^{-n}\,\frac{du}{n\omega_n}\right)^{-\frac1n}\,dt\\
&\ge&\int_0^1\left(\int_{\mathbb S^{n-1}}\Big(\int_{\partial^\star O_t(f)}|u\cdot\nu_{O_t(f)}|\,dH^{n-1}\Big)^{-n}\,\frac{du}{n\omega_n}\right)^{-\frac1n}\,dt\\
&\geq& \inf_{O\in\mathscr{B}(K)}P_{BV,d}(O).
\end{eqnarray*}

On the other hand, assume $\inf_{O\in\mathscr{B}(K)}P_{BV,d}(O)<\infty.$ According to the Definition \ref{BVoriginaldefinition}, for any $L\in\mathscr{B}(K)$ one has
$$
C_{BV,d}(K)\leq \left(\int_{\mathbb S^{n-1}}\Big(\int_{\R^n}|u\cdot\nu_L(x)|\,d|D1_L|\Big)^{-n}\,\frac{du}{n\omega_n}\right)^{-\frac{1}{n}}=P_{BV,d}(L).
$$
This in turn implies
$$C_{BV,d}(K)\leq \inf_{O\in\mathscr{B}(K)}P_{BV,d}(O).$$
Putting the above two cases together, we find \eqref{eq22}.

Now, suppose $K$ is convex. Without loss of generality, we assume $K\in\mathcal K_0^n$. The just-verified formula \eqref{eq22} is utilized to get that for any $\epsilon>0$ there is an $O\in\mathscr{B}(K)$ enjoying
$$
P_{BV,d}(O)<C_{BV,d}(K)+\epsilon.
$$
Notice that
$$\int_{\partial^\star K}|u\cdot\nu_K(x)|\,dH^{n-1}(x)=\int_{K\mid u^\perp}H^0(\partial^\star K\cap(y+L_0^u))\,dH^{n-1}(y)\quad\forall\quad u\in \mathbb S^{n-1}.
$$
Therefore, we have
$$
P_{BV,d}(K)\leq P_{BV,d}(O),
$$
thereby finding
$$
P_{BV,d}(K)<C_{BV,d}(K)+\epsilon\quad\hbox{and\ so}\quad P_{BV,d}(K)\le C_{BV,d}(K).
$$
Meanwhile, it is clear that for an arbitrarily small $\epsilon>0$ one has $1_{(1+\epsilon)K}\in\mathscr{A}(K).$ Therefore,
\begin{eqnarray*}
C_{BV,d}(K) &\leq& \left(\int_{\mathbb S^{n-1}}(\int_{\R^n}|u\cdot\nu_{(1+\epsilon)K}(x)|d|D1_{(1+\epsilon)K}|)^{-n}\,\frac{du}{n\omega_n}\right)^{-\frac{1}{n}}\\
            &=&  \left(\int_{\mathbb S^{n-1}}(\int_{\partial^\star (1+\epsilon)E}|u\cdot\nu_{(1+\epsilon)K}(x)|dH^{n-1})^{-n}\,\frac{du}{n\omega_n}\right)^{-\frac{1}{n}}\\
            &=&P_{BV,d}((1+\epsilon) K)\\
            &=&(1+\epsilon)^{n-1}P_{BV,d}(K).
\end{eqnarray*}
By letting $\epsilon\rightarrow0$, we have
$$
C_{BV,d}(K)\le P_{BV,d}(K),
$$
whence arriving at \eqref{eq23} which is indeed the identity presented in \cite[Theorem 1]{X14r}.

\end{proof}

\subsection{Metric properties of $C_{BV,d}$}\label{s33} Below is a list of the metric properties for the BV-affine capacity with at least two distinctions from $C_{BV}$: First, (vii) indicates that $C_{BV,d}$, as a nonnegative set function on compact sets, is not Choquet capacity, in contrast to the basic fact that $C_{BV}$ is a Choquet capacity; see also \cite[p.457]{Fl} and \cite[Theorem 1.1]{Fr}; and second, the vanishing nature in (viii) is essentially different from the equivalence that $C_{BV}(E)=0$ if and only if $H^{n-1}(E)=0$; see also \cite[Theorem 4.3]{Fl}.

\begin{theorem}\label{t22} The following are valid:

\begin{itemize}
\item[(i)] If $K\subseteq\mathbb R^n$ is compact, then $C_{BV,d}(K)=C_{BV,d}(\partial K)$.

\item[(ii)] If $K\subseteq\mathbb R^n$ is compact and $\{x_0\}+rK=\{x_0+rx:\ x\in K\}\ \forall\ r\in(0,\infty)$, then $C_{BV,d}(\{x_0\}+rK)=r^{n-1}C_{BV,d}(K)$.

\item[(iii)]
For each compact $K\subset \mathbb R^n$, we have $C_{BV,d}(K)=C_{BV,d}(\Phi K+x)$ with $\Phi\in SL(n),$ and $x\in \mathbb R^n.$

\item[(iv)] If $K_1, K_2\subseteq\mathbb R^n$ are compact with $K_1\subseteq K_2\subset \R^n$, then $C_{BV,d}(K_1)\leq C_{BV,d}(K_2).$

\item[(v)] If $\left\{K_j\right\}_{j=1}^{\infty}$ is a sequence of compact subsets of $\mathbb R^n$ with $K_{j+1}\subseteq K_j\subset\R^n$, $j=1,2,3,\cdots$, then
 $C_{BV,d}\left(\bigcap_{j=1}^{\infty}K_j\right)=\lim_{j\rightarrow\infty}C_{BV,d}(K_j)$.

\item[(vi)] For each compact $K\subset \mathbb R^n$ and given $\epsilon>0$ there is an open set $O\supset K$ such that
for every compact $F\subset \mathbb R^n$ with $O\supset F\supset K$ one has $C_{BV,d}(F)\leq C_{BV,d}(K)+\epsilon.$

\item[(vii)] There are two compact sets $E$ and $F$ in $\mathbb R^2$ such that $C_{BV,d}(E\cup F)\ge C_{BV,d}(E)+C_{BV,d}(F)$.

\item[(viii)] $C_{BV,d}(\mathbb B^n)=2\omega_{n-1}$ and $C_{BV,d}(\mathbb B^n\cap \mathbb R^{n-1})=0$.
\end{itemize}
\end{theorem}

\begin{proof} (i) Since $K$ is compact, $\partial K$ is a subset of $K$ and hence $C_{BV,d}(\partial K)\le C_{BV,d}(K)$ due to (iv). To get its reverse inclusion, suppose $L\in\mathscr{B}(\partial K)$. Then $L$ is an open set containing $\partial K$ and hence $\partial K\cup L$ is open and $\partial(K\cup L)$ is a subset of $\partial L$. This, along with \eqref{eq22}, implies $C_{BV,d}(K)\le C_{BV,d}(\partial K)$ and so the formula $C_{BV,d}(K)=C_{BV,d}(\partial K)$ follows.

(ii)-(iv) The results follow from Theorem \ref{t21}(i) and a simple calculation.

(v) It is obvious from the monotonicity of $C_{BV,d}$ that $$C_{BV,d}\left(\bigcap_{j=1}^{\infty}K_j\right)\leq \lim_{j\rightarrow\infty}C_{BV,d}(K_j).$$
So, we are left to show that $$C_{BV,d}\left(\bigcap_{j=1}^{\infty}K_j\right)\geq \lim_{j\rightarrow\infty}C_{BV,d}(K_j).$$
We know from the definition of $C_{BV,d}(\cdot)$ that, there exists an $f\in BV(\R^n),$ with
$\bigcap_{j=1}^{\infty}K_j$ being contained in the interior $\{x:f(x)\geq 1\}^\circ
$ of $\{x:f(x)\geq 1\}$ such that $$\left(\left(n\omega_n\right)^{-1}\int_{\mathbb S^{n-1}}(\int_{\R^n}|u\cdot\sigma_f(x)|d|Df|)^{-n}du\right)^{-\frac{1}{n}}\leq C_{BV,d}\left(\bigcap_{j=1}^{\infty}K_j\right)+\epsilon.$$

Since $K_j$ is compact, we must have $$K_j\subseteq\{x:f(x)\geq1\}^\circ$$
for $j$ big enough. Therefore, $$C_{BV,d}(K_j)\leq \left(\left(n\omega_n\right)^{-1}\int_{\mathbb S^{n-1}}(\int_{\R^n}|u\cdot\sigma_f(x)|d|Df|)^{-n}du\right)^{-\frac{1}{n}}.$$
Therefore, $$C_{BV,d}\left(\bigcap_{j=1}^{\infty}K_j\right)\geq \lim_{j\rightarrow\infty}C_{BV,d}(K_j).$$

(vi) According to Theorem \ref{t21}(ii), there is an open set $L\supseteq K$ such that
$$
P_{BV,d}(L)\leq C_{BV,d}(K)+\epsilon.
$$ Consequently,
$$
C_{BV,d}(F)\leq P_{BV,d}(L)\quad\forall\quad F\subset L.
$$
Therefore,
$$
C_{BV,d}(F)\leq C_{BV,d}(K)+\epsilon.
$$

(vii) Let
$$
\begin{cases}
E=\{(x,y)\in\mathbb R^2:-500\leq x\leq 500 \mbox{~and~} -5\leq y\leq 5\};\\ F=\{(x,y)\in\mathbb R^2:-5\leq x\leq 5 \mbox{~and~} -500\leq y\leq 500\}.
\end{cases}
$$
Note first that the affine invariance of $C_{BV,d}$ are used to yield
$$
C_{BV,d}(E)=C_{BV,d}(\Phi E)\ \ \mbox{~where~}\ \ \Phi=
\begin{pmatrix} \frac{1}{10} & 0 \\ 0 & 10 \end{pmatrix}.
$$
So we have according to the monotonicity of $C_{BV,d}(\cdot)$
$$
C_{BV,d}(50\sqrt{2}\mathbb B^n)\geq C_{BV,d}(\Phi E)\geq C_{BV,d}(50\mathbb B^n).
$$
For $C_{BV,d}(E\cup F),$ we recall that from Theorem \ref{t21}(ii), we get that there is an open set $L$ such that
$$
E\cup F\subset L\quad\hbox{and}\quad
P_{BV,d}(L)\leq C_{BV,d}(E\cup F)+\epsilon.
$$
But for $P_{BV,d}(L)$, we know
\begin{eqnarray*}
P_{BV,d}(L)&=&\left(\int_{\mathbb S^{1}}(\int_{\partial^{\star}L}|u\cdot\nu_L(x)|dH^{1}(x))^{-2}\,\frac{du}{\sigma_1}\right)^{-\frac{1}{2}}\\
&=&\left(\int_{\mathbb S^{1}}\Big(\int_{L|u^{\perp}}H^\circ(\partial^{\star}L\cap(y+L_0^u))\, dH^{1}(y)\Big)^{-2}\,\frac{du}{\sigma_1}\right)^{-\frac{1}{2}}
\end{eqnarray*}
Notice that
$$
H^1(L|u^{\perp})\geq H^1(E\cup F|u^{\perp})\geq 1000\ \ \&\ \ H^\circ(\partial^{\star}L\cap(y+L_0^u))\geq 2
$$
for $y\in L|u^{\perp}.$ Thus, upon setting
$$
G=\{(x,y)\in\mathbb R^2:|x|+|y|<500\}
$$
we have
$$
C_{BV,d}(G)\geq C_{BV,d}\left(\frac{500}{\sqrt{2}}B_2\right).
$$
So, we will have
\begin{eqnarray*}
C_{BV,d}(E\cup F)+\epsilon &\geq& P_{BV,d}(L)   \\&\geq& P_{BV,d}(G)\\
&\geq& (\frac{500}{\sqrt{2}})C_{BV,d}(\mathbb B^n)\\ &\geq& 2\times(50\sqrt{2})C_{BV,d}(\mathbb B^n)\\&\geq& C_{BV,d}(E)+C_{BV,d}(F).
\end{eqnarray*}

(viii) This follows from \cite{X15} and \cite[Theorem 2.4(i)]{XZ} under $p=1$.

\end{proof}

\section{Steiner's symmetrization for the affine BV-capacity}\label{s4}

\subsection{A known assertion}\label{s41} We first recall the following two facts: The first (i) is from \cite{LYZ00}; and the second (ii) is from \cite{Wang}.

\begin{lemma}\label{l31} The following are valid:

\begin{itemize}
\item[(i)] Suppose $K,L\subset\mathbb R^{n-1}\times\mathbb R$ are convex bodies. Then $S_{e_n}\left(K^{\circ}\right)\subseteq L^{\circ}$
if and only if
\begin{equation}
\label{eq31}
h_K(x,t)=1=h_K(x,-s)\ \ \mbox{~with~}  t\neq -s\Rightarrow h_L\big(x,\frac{1}{2}t+\frac{1}{2}s\big)\leq 1.
\end{equation}

\item[(ii)] If $\mathcal P$ is the family of sets of finite perimeter, and $P_{BV,d}:\mathcal P\rightarrow \mathbb R$ is the affine perimeter functional, then $P_{BV,d}(\cdot)$ is continuous in the strict convergence topology of $BV(\mathbb R^n).$
\end{itemize}
\end{lemma}

\subsection{Decreasing under Steiner's symmetrization and rounding}\label{s42} The following two inequalities under the Steiner symmetrization are well-known for $P_{BV}$ and $C_{BV}$; see e.g. \cite{Sa}.

\begin{theorem}\label{t31} Let $E$ be a compact set in $\mathbb R^n$ and $u\in \mathbb S^{n-1}$. If $S_{u}(E)$ is the Steiner symmetrization of $E$ in the direction $u$, then
\begin{equation}
\label{eq32}
P_{BV,d}(S_{u}(E))\le P_{BV,d}(E)\quad\&\quad
C_{BV,d}(S_{u}(E))\le C_{BV,d}(E).
\end{equation}
\end{theorem}
\begin{proof}  According to \eqref{eq22}, it is enough to verify the first inequality of \eqref{eq32}. Without loss of generality, we assume $u=e_n.$
If $f:\mathbb R^n\to\mathbb R$ and $G\subset \mathbb R^{n-1}$, we denote the graph of $f$ over $G$ by
$$
\Gamma(f,G)=\{(z,t)\in\mathbb R^n:z\in G,t=f(z)\}.
$$

Note that if $P_{BV,d}(E)=\infty$, then the statement is trivial. So we focus on the case when $P_{BV,d}(E)$ is finite.

Let's first assume that $E$ is a compact set with polyhedral boundary and that the outer unit normal to $E$ is never orthogonal to $e_n.$

By the assumption, and by the implicit function theorem, there exists a partition of the set
$$
G=\{z\in\mathbb R^{n-1}:L^1(E\cap L_z^{e_n})>0\}
$$
into finitely many $(n-1)$-dimensional polyhedral sets $\{G_h\}_{h=1}^M$ in $\mathbb R^{n-1}$,
$$
G=\bigcup_{h=1}^{M}G_h,
$$
 and affine functions
$$
g_h^k,f_h^k:G_h\to\mathbb R\quad\forall\quad (h,k)\in [1,\leq M]\times[1,\leq N(h)],
$$
with
$$
\partial E=\bigcup_{h=1}^{M}\bigcup_{k=1}^{N(h)}\Gamma(f_h^k,G_h)\cup\Gamma(g_h^k,G_h),
$$
$$
E=\bigcup_{h=1}^{M}\left\{(z,t)\in G_h\times\mathbb R:t\in\bigcup_{k=1}^{N(h)}\big(g_h^k(z),f_h^k(z)\big)\right\};
$$
thus, if
$$
m(z)=L^1\big(E\cap L_z^{e_n}\big)\quad\forall\quad z\in\mathbb R^{n-1}\times\left\{0\right\},
$$
then
$$
m(z)=\sum_{k=1}^{N(h)}f_h^k(z)-g_h^k(z),~~~~~\forall z\in G_h,
$$
so that $m$ is affine on each $G_h.$ Moreover, $m$ is continuous and piece-wise affine on each $\mathbb R^{n-1}.$ Since, $S_{e_n}\left(E\right)=\{(z,t)\in G\times \mathbb R:|t|<\frac{m(z)}{2}\},$ $S_{e_n}\left(E\right)$ is a bounded open set with polyhedral boundary.

Suppose
$$
h_{\Pi E}(x,t)=1=h_{\Pi E}(x,-s),
$$
recall that at the point $y=(x,f_j(x)),$ the outer normal to $\Gamma(f_j,G_i)$ is
$$
\nu_E(y)=\frac{(-\nabla f_j(x),1)}{\sqrt{|\nabla f_j(x)|^2+1}},
$$
and the outer normal to $\Gamma\big(g_j,G_i\big)$ at the point $y=(x,g_j(x))$ is
$$
\nu_E(y)=\frac{(\nabla g_j(x),-1)}{\sqrt{|\nabla g_j(x)|^2+1}}.
$$
So we have
\begin{align*}
h_{\Pi E}(x,t)&=\frac{1}{2}\int_{\partial^\star E}|(x,t)\cdot\nu_E(y)|dH^{n-1}(y)\\
&=\frac{1}{2}\int_{\partial E}|(x,t)\cdot\nu_E(y)|dH^{n-1}(y)\\
&= \sum_{i=1}^M\int_{G_i} \left(\Big(\sum_{j=1}^{N(i)}\Big|\frac{t-x\cdot \nabla f_j}{\sqrt{1+|\nabla f_j|^2}}\Big|\Big){\sqrt{1+|\nabla f_j|^2}}+\Big(\sum_{j=1}^{N(i)}\Big|\frac{-t+x\cdot\nabla g_j}{\sqrt{1+|\nabla g_j|^2}}\Big|\Big)\sqrt{1+|\nabla g_j|^2}\right)dx\\
&=\sum_{i=1}^M\int_{G_i} \left(\sum_{j=1}^{N(i)}|{t-x\cdot \nabla f_j}|+\sum_{j=1}^{N(i)}|{-t+x\cdot\nabla g_j}|\right)\,dx.
\end{align*}
Similarly, we have
$$
h_{\Pi E}(x,-s)=\sum_{i=1}^M\int_{G_i} \left(\sum_{j=1}^{N(i)}|{s+x\cdot \nabla f_j}|+\sum_{j=1}^{N(i)}|{s+x\cdot\nabla g_j}|\right)dx.
$$
Utilizing the above calculation, we get
\begin{align*}
h_{\Pi S_{e_n}E}\left(x,\frac{s+t}{2}\right)&=\sum_{i=1}^M\int_{G_i}\left(\Big|{\frac{s+t}{2}-x\cdot \nabla \frac{\sum_{j=1}^{N(i)}(f_j-g_j)}{2}}\Big|+\Big|{-\frac{s+t}{2}-x\cdot \nabla \frac{\sum_{j=1}^{N(i)}(f_j-g_j)}{2}}\Big|\right)dx\\
&\leq \frac{1}{2}h_{\Pi E}(x,t)+\frac{1}{2}h_{\Pi E}(x,-s)=1.
\end{align*}
Therefore, according to \eqref{eq31} we have
$$
S_{e_n}\big(\Pi^{\circ}E\big)\subset \Pi^{\circ}S_{e_n}\big(E\big).
$$
Recalling
$$
P_{BV,d}(E)=\big({2^n\omega_n}\big)^{\frac{1}{n}}\big(V(\Pi^\circ E)\big)^{-\frac1n},
$$
we get
$$
P_{BV,d}\big(S_{e_n}(E)\big)\leq P_{BV,d}(E).
$$

If $E$ is an arbitrary compact set of finite perimeter, then we choose $E_j$ to be an approximation sequence to $E$ in the strict convergence topology, with $E_j$ being a sequence of compact set with polyhedral boundary and that the outer unit normal to $E_j$ is never orthogonal to $e_n.$

We notice that the sets $S_{e_n}(E_k)$ are equibounded in $BV(\mathbb R^n).$ Therefore there are some set function $v\in BV(\mathbb R^n)$ and a subsequence which we still denote by $S_{e_n}(E_{k})$ such that
$$
1_{S_{e_n}(E_{k})}\rightharpoonup v\quad \mbox{~weakly\ \ in~}\quad BV(\mathbb R^n).
$$
On the other hand, from the fact that the $n$-dimensional Lebesgue measure of $S_{e_n}(E_k)\bigtriangleup S_{e_n}(E)$ is not greater than the $n$-dimensional Lebesgue measure of $E_k\bigtriangleup E$ it follows that
$$
1_{S_{e_n}(E_k)}\rightarrow 1_{S_{e_n}(E)}\quad \mbox{~in~}\quad L^1(\mathbb R^n).
$$
Now let $\mu_i$ denote the Radon-measure which is associated with the weak partial derivative $v_i (i=1,\ldots,n).$ Then we have
$$
\int_{\mathbb R^n}\phi\frac{\partial 1_{S_{e_n}(E_k)}}{\partial x_i}dx \longrightarrow \int_{\mathbb R^n}\phi d\mu_i\quad\forall\quad \phi \in C_c^{\infty}(\mathbb R^n)
$$
while, we have
\begin{eqnarray*}
\int_{\mathbb R^n}\phi\frac{\partial 1_{S_{e_n}(E_k)}}{\partial x_i}dx=-\int_{\mathbb R^n}\partial_{x_i}\phi1_{S_{e_n}(E_k)}dx\longrightarrow -\int_{\mathbb R^n}\partial_{x_i}\phi1_{S_{e_n}(E)}dx
\end{eqnarray*}
thereby finding $v=1_{S_{e_n}(E)}.$ Consequently, the Reshetnyak continuity theorem (cf. \cite[p.269] {Setsoffiniteperimeterandgeometricvariational}) is utilized to derive
$$
\int_{\partial^\star S_{e_n}(E_k)}|u\cdot\nu_{S_{e_n}(E_k)}(x)|dH^{n-1}(x)\longrightarrow\int_{\partial^\star S_{e_n}(E)}|u\cdot\nu_{S_{e_n}(E)}(x)|dH^{n-1}(x)\quad \mbox{~pointwisely~for~}\quad u\in \mathbb S^{n-1}.
$$
According to Lemma 4.1 of \cite{Wang}, we have
$$
\int_{\partial^\star S_{e_n}(E_k)}|u\cdot\nu_{S_{e_n}(E_k)}(x)|dH^{n-1}(x)\longrightarrow\int_{\partial^\star S_{e_n}(E)}|u\cdot\nu_{S_{e_n}(E)}(x)|dH^{n-1}(x)\quad \mbox{uniformly},
$$
whence getting
$$
P_{BV,d}\big(S_{e_n}(E_k)\big)\longrightarrow P_{BV,d}\big(S_{e_n}(E)\big).$$ Utilizing Lemma \ref{l31}(ii), we obtain
$$
P_{BV,d}\big(S_{e_n}(E)\big)=\lim_{n\rightarrow\infty}P_{BV,d}\big(S_{e_n}(E_n)\big)\leq\lim_{n\rightarrow\infty}P_{BV,d}(E_n)=P_{BV,d}(E).
$$
 \end{proof}

The following principle corresponds to the well-known fact on $P_{BV}$ (cf. \cite{Gi}) and $C_{BV}$ (cf. \cite{Sa}).

\begin{theorem}\label{t32}
If $R(E)$ is the rounding of $E$, then
\begin{equation}
\label{eq33}
P_{BV,d}(R(E))\le P_{BV,d}(E)\quad\&\quad
C_{BV,d}(R(E))\le C_{BV,d}(E).
\end{equation}
\end{theorem}
\begin{proof}
Since $R(E)$ is the rounding of $E$, $R(E)$ is a ball with its volume being the same as $V(E)$. Now, an application of \cite[Theorem 7.2]{Wang} gives
\begin{align*}
\big(V(E)\big)^{n-1}&\int_{\mathbb S^{n-1}}\left(2^{-1}\int_{\partial^\star E}|u\cdot\nu_E|\,dH^{n-1}\right)^{-n}\,\frac{du}{n}\\
&=\big(V(E)\big)^{n-1}V(\Pi^\circ E)\\
&\le (\omega_n\omega^{-1}_{n-1})^n\\
&=\big(V(R(E))\big)^{n-1}V(\Pi^\circ R(E))\\
&=\big(V(R(E))\big)^{n-1}\int_{\mathbb S^{n-1}}\left(2^{-1}\int_{\partial^\star R(E)}|u\cdot\nu_{R(E)}|\,dH^{n-1}\right)^{-n}\,\frac{du}{n}.
\end{align*}
This last inequality implies
$$
\int_{\mathbb S^{n-1}}\left(2^{-1}\int_{\partial^\star E}|u\cdot\nu_E|\,dH^{n-1}\right)^{-n}\,\frac{du}{n}\le\int_{\mathbb S^{n-1}}\left(2^{-1}\int_{\partial^\star R(E)}|u\cdot\nu_{R(E)}|\,dH^{n-1}\right)^{-n}\,\frac{du}{n}
$$
whence deriving
$$
P_{BV,d}(R(E))\le P_{BV,d}(E),
$$
and by \eqref{eq22}
$$
C_{BV,d}(R(E))\le C_{BV,d}(E).
$$
\end{proof}
\section{Traces decided by the affine BV-capacity}\label{s5}

\subsection{An affine trace inequality}\label{s51} By an affine trace inequality we mean an inequality of the form
\begin{equation}
\label{e51}
\sup_{0\not\equiv f\in BV(\R^n)}\frac{\|f\|_{L^q_\mu(\R^n)}}{\|f\|_{BV,d(\R^n)}}<\infty,
\end{equation}
where $L^q_\mu(\R^n)$ is the Lebesgue $q$-space with respect to a given Radon measure on $\R^n$. As shown in the coming-up next assertion which may be regarded as an affine counterpart of \cite[Theorem 5.12.4]{Zi} or \cite[Theorem 4.7]{MZ}, the validity of \eqref{e51} is totally determined by restricting $\mu$ on a lower dimensional submanifold of $\R^n$.
\begin{theorem}
\label{t41} Given $\frac{n}{n-1}\ge q\ge 1$ and a nonnegative Radon measure $\mu$ on $\mathbb R^n$. The following three statements are equivalent.

\begin{itemize}

\item[(i)] There is a constant $\kappa_1>0$ such that $\Big(\int_{\R^n}|f|^q\,d\mu\Big)^\frac1q\le \kappa_1\|f\|_{BV,d(\R^n)}\ \forall\ f\in BV(\R^n)$.

\item[(ii)] There is a constant $\kappa_2>0$ such that $\big(\mu(K)\big)^\frac1q\le \kappa_2 C_{BV,d}(K)\ \forall\ \hbox{compact}\ K\subseteq\mathbb R^n$.

\item[(iii)] There is a constant $\kappa_3>0$ such that $\big(\mu(O)\big)^\frac1q\le \kappa_3 P_{BV,d}(O)\ \forall\ \hbox{bounded\ open}\ O\subseteq\mathbb R^n$.

\end{itemize}
\end{theorem}
\begin{proof} Three implications are treated below.

(i)$\Rightarrow$(ii) This follows from Definition \ref{BVoriginaldefinition}.

(ii)$\Rightarrow$(iii) This follows from Theorem \ref{t21}(ii) and the outer regularity of $\mu$.

(iii)$\Rightarrow$(i) Suppose (iii) is true. Because any $BV(\R^n)$-function can be approximated by $C^\infty_c(\R^n)\cap BV(\R^n)$-functions in the strict convergence topology, in what follows we only consider $f\in C^\infty_c(\R^n)\cap BV(\R^n)$. Such a function $f$ ensures that
$$
t\mapsto\mu\big(\{x\in\R^n: |f(x)|>t\}\big)=:\mu_t(f)
$$
is a decreasing function on $(0,\infty)$. Thus, we have the following inequality (cf. Adams' inequality in \cite[Lemma 1.86]{MZ1}):
$$
t^q\mu_t(f)\le\left(\int_0^t (s^q\mu_s(f))^\frac1q\,\frac{ds}{s}\right)^q\quad\forall\quad t\in (0,\infty).
$$
Using the layer-cake formula and the argument for Theorem \ref{t21}(ii) we get
\begin{align*}
\int_{\R^n}|f|^q\,d\mu&=q\int_0^\infty\big(\mu_t(f)\big)^\frac1q\big((\mu_t(f))^\frac1q t\big)^{q-1}\,dt\\
&\le q\left(\int_0^\infty \big(\mu_t(f)\big)^\frac1q\Big(\int_0^t\big(s^q\mu_s(f)\big)^\frac1q\,\frac{ds}{s}\Big)^{q-1}\,dt\right)\\
&\le q\left(\int_0^\infty\big(\mu_s(f)\big)^\frac1q\,ds\right)^q\\
&\le q\kappa_3^q\left(\int_0^\infty  P_{BV,d}\big(\{x\in\R^n: |f(x)|> t\}\big)\,dt\right)^q\\
&\le q\kappa_3^q\|f\|^q_{BV,d(\R^n)},
\end{align*}
whence reaching (i).
\end{proof}

Here, it is appropriate to point out that only (iii)$\Rightarrow$(i) provides no a sharp constant unless $q=1$. But, an improved argument for this implication under the case $\mu=H^n=V$ and $q=n/(n-1)$ can be given via \cite[Theorem 6.4]{Wang}, thereby verifying \eqref{e05}$\Rightarrow$\eqref{e04} and consequently \eqref{e05} $\Leftrightarrow$\eqref{e07} for any compact set.

\subsection{The affine $q$-Cheeger constant}\label{s52} A further look at the case $\mu=V=H^n$ of Theorem \ref{t41} reveals that (ii) or (iii) in Theorem \ref{t41} is valid only for $q=n/(n-1)$ which induces the affine isoperimetry/Sobolev constant:

\begin{equation}\label{eq41}
\inf\left\{\frac{P_{BV,d}(O)}{\big(V(O)\big)^\frac{n-1}{n}}: \hbox{bounded\ open}\ {O}\ \hbox{with}\ C^1\ \hbox{boundary}\ \partial{O}\right\}=\frac{2\omega_{n-1}}{\omega_n^\frac{n-1}{n}},
\end{equation}
where the infimum in \eqref{eq41} is attainable by any origin-symmetric ellipsoid in $\R^n$ according to the equality case of \cite[Theorem 7.2]{Wang}.

Due to
$$
V(r\mathbb B^n)=\omega_n r^n\quad\&\quad P_{BV,d}(r\mathbb B^n)=2\omega_{n-1}r^{n-1},
$$
we are suggested to consider the so-called affine $[1,\frac{n}{n-1})\ni q$-Cheeger constant of an open set $O\subseteq\R^n$ with $V(O)<\infty$:
\begin{equation}
\label{eq42}
\mathsf{h}_{q,d}(O):=\inf_{D\subset O}\frac{P_{BV,d}(D)}{\big(V(D)\big)^\frac1q}\le\left(\frac{2\omega_{n-1}}{n\omega_n}\right)\inf_{D\subset O}\frac{P_{BV}(D)}{\big(V(D)\big)^\frac1q}=:\left(\frac{2\omega_{n-1}}{n\omega_n}\right)\mathsf{h}_q(O),
\end{equation}
where $\mathsf{h}_q(O)$ is the so-called $q$-Cheeger constant of $O$; see also \cite{Pf} and \cite{Cheeger} for the root of an original Cheeger constant. If $E\subseteq O$ solves
$$
 \mathsf{h}_{q,d}(O)=\frac{P_{BV,d}(E)}{\big(V(E)\big)^\frac1q},
$$
then $E$ is called an affine $q$-Cheeger set. Although there is always a $q$-Cheeger set as proved in \cite[Theorem 2.2]{Pf}, we can only obtain the following result on an affine $q$-Cheeger set.

\begin{proposition}
\label{t42} Let $1\le q<n/(n-1)$ and $O$ be an open subset of $\R^n$ with $V(O)<\infty$. If $\mathsf{h}_{q,d}(O)$ is attainable by a set $E\subseteq O$, then $\partial E\cap\partial O\not=\emptyset$.
\end{proposition}
\begin{proof} Without loss of generality, we may assume $\mathsf{h}_{q,d}(O)<\infty$ - otherwise any subset of $O$ is an affine $q$-Cheeger set and consequently there is nothing to argue. Since $E$ is an affine $q$-Cheeger set, by definition there is a sequence of subsets $D_j$ of $O$ such that
$$
\lim_{j\to\infty}\frac{P_{BV,d}(D_j)}{\big(V(D_j)\big)^\frac1q}=\mathsf{h}_{q,d}(O)=\frac{P_{BV,d}(E)}{\big(V(E)\big)^\frac1q}.
$$
Consequently, $V(E)>0$. In fact, if $V(E)=0$, then $P_{BV,d}(E)=0$ and so $V(D_j)\to 0$ due to $\mathsf{h}_{q,d}(O)<\infty$, and hence via \eqref{eq33} in Theorem \ref{t32} for $P_{BV,d}(D_j)$, i.e., via finding a ball $r_{D_j}\mathbb B^n$ with radius $r_{D_j}$ such that
$$
V(D_j)=V(r_{D_j}\mathbb B^n)\quad\&\quad P_{BV,d}(D_j)\ge P_{BV,d}(r_{D_j}\mathbb B^n)
$$
we use the hypothesis $q\in [1,n/(n-1))$ to obtain
$$
\infty>\mathsf{h}_{q,d}(O)\leftarrow\frac{P_{BV,d}(D_j)}{\big(V(D_j)\big)^\frac1q}\ge \frac{P_{BV,d}(r_{D_j}\mathbb B^n)}{\big(V(r_{D_j}\mathbb B^n)\big)^\frac1q}=\left(\frac{2\omega_{n-1}}{\omega_n}\right) r_j^{n(\frac{n-1}{n}-\frac{1}{q})}\to\infty\quad \hbox{as}\quad r_j\to 0,
$$
whence producing a contradiction. Thus, $V(E)>0$.

Next, in order to see $\partial E\cap\partial O\not=\emptyset$, let us assume $E\Subset O$. Then there would be an $r>1$ such that
$$
r E:=\{rx:\ x\in E\}\subseteq O,
$$
and hence
$$
\mathsf{h}_{q,d}(O)\le\frac{P_{BV,d}(r E)}{\big(V(r E)\big)^\frac1q}=\frac{r^{n-1}P_{BV,d}(E)}{r^{\frac{n}{q}} \big(V(E)\big)^{\frac1q}}<\frac{P_{BV,d}(E)}{\big(V(E)\big)^{\frac1q}}=\mathsf{h}_{q,d}(O)
$$
contradicting the assumption that $E$ is an affine $q$-Cheeger set.
\end{proof}

The next assertion connects the affine Cheeger constant and the so-called affine BV-Sobolev constant (or the generalized affine eigenvalue); see e.g. \cite{XZ, BBF, KF}.

\begin{theorem}
\label{t43} Given an open set $O\subseteq\mathbb R^n$ with $V(O)<\infty$, $p\in [1,\frac{q}{q-1})$ and $q\in [1,\frac{n}{n-1})$, let $W^{1,p}_0(O)$ be the completion of $C^\infty$-functions $f$ with compact support in $O$ (i.e. $f\in C^\infty_c(O)$) under the $p$-Sobolev norm $\|f\|_{L^p(O)}+\big\||\nabla f|\big\|_{L^p(O)}$. For $f\in W^{1,p}_0(O)$ set
$$
\|f\|_{\dot{W}^{1,p}_d(O)}=\left(\int_{\mathbb S^{n-1}}\big\||\nabla_u f|\big\|_{L^p(O)}^{-{n}}\,\frac{du}{n\omega_n}\right)^{-\frac{1}{n}}.
 $$
Then
\begin{equation}\label{eq43}
\mathsf{h}_{q,d}(O)\le {q^\frac{1-q}{q}\left(\frac{pq}{p-(p-1)q}\right)}\inf_{f\in W^{1,p}_0(O)\setminus\{0\}}\frac{\|f\|_{\dot{W}_d^{1,p}(O)}}{\|f\|_{L^\frac{pq}{p-(p-1)q}(O)}}=:\Lambda_{p,q,d}(O).
\end{equation}
Especially
\begin{equation}
\label{eq44}
\Lambda_{1,1,d}(O)=\mathsf{h}_{1,d}(O).
\end{equation}
\end{theorem}

\begin{proof} For $f\in C^\infty_c(O)$ and $t>0$ let $O_t(f)=\{x\in O:\ |f(x)|>t\}$ and $u\in\mathbb S^{n-1}$. Then an application of the co-area formula gives
$$
\int_O|\nabla_u f(x)|\,dx=\int_0^\infty\left(\int_{\partial O_t(f)}\Big|u\cdot\frac{\nabla f}{|\nabla f|}\Big|\,dH^{n-1}\right)\,dt.
$$
Note that the projection body $\Pi O_t(f)$ of $O_t(f)$ is a convex set determined by the support function
$$
h_{\Pi O_t(f)}(u)=2^{-1}\int_{\partial^\star O_t(f)}|u\cdot\nu_{O_t(f)}|\,dH^{n-1}\quad\forall\quad u\in\mathbb S^{n-1}.
$$
So, the above two formulas, the Minkowski inequality are utilized and the arguments for Theorem \ref{t41}(iii)$\Rightarrow$(i) are used to imply
\begin{align*}
\|f\|_{\dot{W}^{1,1}_d(O)}&=\left(\int_{\mathbb S^{n-1}}\Big(\int_0^\infty 2h_{\Pi O_t(f)}(u)\,dt\Big)^{-n}\,\frac{du}{n\omega_n}\right)^{-\frac1n}\\
&\ge \int_0^\infty\Big(\int_{\mathbb S^{n-1}}\big(2h_{\Pi O_t(f)}(u)\big)^{-n}\,\frac{du}{n\omega_n}\Big)^{-\frac1n}\,dt\\
&=\int_0^\infty \big(V(O_t(f))\big)^{-\frac1q}\Big(\int_{\mathbb S^{n-1}}\big(2h_{\Pi O_t(f)}(u)\big)^{-n}\,\frac{du}{n\omega_n}\Big)^{-\frac1n} \big(V(O_t(f))\big)^\frac1q\,dt\\
&=\int_0^\infty \big(V(O_t(f))\big)^{-\frac1q}P_{BV,d}(O_t(f)) \big(V(O_t(f))\big)^\frac1q\,dt\\
&\ge \mathsf{h}_{q,d}(O)\int_0^\infty \big(V(O_t)\big)^{\frac1q}\,dt\\
&\ge q^{-\frac1q}\mathsf{h}_{q,d}(O)\|f\|_{L^q(O)}.
\end{align*}
Since $C^\infty_c(O)$ is dense in $W^{1,1}_0(O)$, this last estimation is valid for any $f\in W^{1,1}_0(O)$. When
$$
\begin{cases}
p\in [1,q/(q-1));\\
q\in [1,n/(n-1));\\
\alpha=p/(p-(p-1)q);\\
g\in W^{1,p}_0(O),
\end{cases}
$$
we use the H\"older inequality to obtain
$$
\int_O \big|\nabla_u(|g|^{\alpha-1}g)|\,dx=\alpha\int_O |g|^{\alpha-1}|\nabla_u g|\,dx\le \alpha\left(\int_O|g(x)|^\frac{(\alpha-1)p}{p-1}dx\right)^\frac{p-1}{p}\|\nabla_u g\|_{L^p(O)},
$$
thereby finding
$$
\big\||g|^{\alpha-1}g\big\|_{\dot{W}^{1,1}_d(O)}\le \alpha\left(\int_O|g(x)|^\frac{(\alpha-1)p}{p-1}dx\right)^\frac{p-1}{p}\|g\|_{\dot{W}^{1,p}_d(O)}.
$$
Consequently, we substitute $|g|^{\alpha-1}g$ for $f$ in the above $\|f\|_{\dot{W}^{1,1}_d(O)}$-estimation to gain
$$
\alpha\left(\int_O|g(x)|^\frac{(\alpha-1)p}{p-1}dx\right)^\frac{p-1}{p}\|g\|_{\dot{W}^{1,p}_d(O)}\ge q^{-\frac1q}\big\||g|^\alpha\big\|_{L^q(O)}\mathsf{h}_{q,d}(O),
$$
which, plus $\alpha=p/(p-(p-1)q)$, gives
$$
\mathsf{h}_{q,d}(O)\le q^\frac{1-q}{q}\left(\frac{pq}{p-(p-1)q}\right)\left(\frac{\|g\|_{\dot{W}_d^{1,p}(O)}}{\|g\|_{L^\frac{pq}{p-(p-1)q}(O)}}\right).
$$
Taking the infimum of the last inequality over $g\in W^{1,p}_0(O)\setminus\{0\}$, we obtain \eqref{eq43}.

Note that \eqref{eq43} under $p=q=1$ gives
$$
\Lambda_{1,1,d}(O)\ge\mathsf{h}_{1,d}(O).
$$
Accordingly, it remains to verify the reverse form of the last inequality. To do so, for any natural number $j>1$ we select an $O_j\subseteq O$ such that
$$
{P_{BV,d}(O_j)}\big({V(O_j)}\big)^{-1}<\mathsf{h}_{1,d}(O)+j^{-1}.
$$
Now, given an arbitrarily small $\epsilon>0$ let
$$
\begin{cases}
f_\epsilon(x)=\begin{cases}
0\quad\hbox{as}\quad x\in O\ \&\ \hbox{dist}(x,O_j)\ge\epsilon\\
1-\epsilon^{-1}\hbox{dist}(x,O_j)\quad\hbox{as}\quad x\in O\ \&\  \hbox{dist}(x,O_j)<\epsilon;
\end{cases}\\
O_{j,\epsilon}=\{x\in O: 0<\hbox{dist}(x,O_j)<\epsilon\}.
\end{cases}
$$
Then (cf. \cite[p.195]{Z})
$$
\int_O|\nabla_u f_\epsilon(x)|\,dx=\epsilon^{-1}\int_{O_{j,\epsilon}}|u\cdot\nu(x')|\,dx\quad\forall\quad u\in\mathbb S^{n-1},
$$
where $x'\in\partial O_j$ satisfies
$$
\hbox{dist}(x,O_j)=|x'-x|\quad\&\quad \nu(x')=|x'-x|^{-1}(x'-x).
$$
Note that if $h_{\Pi Q_j}(u)$ is the projection function of $O_j$ as usual, then
$$
\lim_{\epsilon\to 0} \epsilon^{-1}\int_{O_{j,\epsilon}}|u\cdot\nu(x')|\,dx=2h_{\Pi Q_j}(u).
$$
So, a further computation gives
$$
\Lambda_{1,1,d}(O)\le \frac{\|f_\epsilon\|_{\dot{W}^{1,1}(O)}}{\|f_\epsilon\|_{L^1(O)}}\to \frac{P_{BV,d}(O_j)}{V(O_j)}<\mathsf{h}_{1,d}(O)+j^{-1}\quad\hbox{as}\quad\epsilon\to 0.
$$
Now, letting $j\to\infty$ in the above estimation produces
$$
\Lambda_{1,1,d}(O)\le\mathsf{h}_{1,d}(O),
$$
and then \eqref{eq44} follows.
\end{proof}

\end{CJK*}

\end{document}